\documentclass[11pt]{amsart}
\usepackage[margin=1in]{geometry}                
\usepackage[parfill]{parskip}    
\usepackage{graphicx}
\usepackage{amssymb}
\usepackage{epstopdf}
\usepackage{latexsym,enumitem,amsmath,amsthm,amsfonts,bbm,hyperref,bm,stmaryrd,tikz-cd}
\usepackage{mathabx}
\usepackage{url}
\usepackage{comment}

\newtheorem{theorem}{Theorem}
\newtheorem{corollary}[theorem]{Corollary}
\newtheorem{lemma}[theorem]{Lemma}
\newtheorem{proposition}[theorem]{Proposition}
\newtheorem*{lem:kirkgoebel}{Lemma \ref{lem:kirkgoebel}}

\theoremstyle{definition}
\newtheorem{definition}[theorem]{Definition}
\newtheorem{remark}[theorem]{Remark}

\newtheorem{example}[theorem]{Example}

\makeatletter
\newcommand{\dotminus}{\mathbin{\text{\@dotminus}}}

\newcommand{\@dotminus}{%
  \ooalign{\hidewidth\raise1ex\hbox{.}\hidewidth\cr$\m@th-$\cr}%
}
\makeatother

\newcommand{\Vop}{\mathcal{V}^{\text{op}}}

\newcommand{\Rop}{\mathbb{R}^{\text{op}}}

\newcommand{\mcv}{\mathcal{V}}

\newcommand{\PSh}{\operatorname{PSh}}
\newcommand{\sPSh}{\operatorname{sPSh}}

\newcommand{\Set}{\mathbf{Set}}

\newcommand{\Met}{\mathbf{Met}}

\newcommand{\sSet}{\mathbf{sSet}}

\newcommand{\yon}{\mathbf{y}}

\newcommand{\Dop}{\Delta^{\text{op}}}
\newcommand{\vgph}{\mathcal{V}\text{-}\mathbf{Gph}}
\newcommand{\rgph}{\mathbb{R}\text{-}\mathbf{Gph}}
\newcommand{\vocat}{(\mathcal{V},\otimes)\text{-}\mathbf{Cat}}

\newcommand{\rpcat}{(\mathbb{R},+)\text{-}\mathbf{Cat}}
\newcommand{\vocato}{(\mathcal{V},\otimes_1)\text{-}\mathbf{Cat}}
\newcommand{\vocatt}{(\mathcal{V},\otimes_2)\text{-}\mathbf{Cat}}
\newcommand{\Fr}{\operatorname{Free}}
\newcommand{\Gn}{\Gamma^n(r_1, \dots, r_n)}
\newcommand{\Don}{\Delta^n_{\otimes}(r_1, \dots, r_n)}
\newcommand{\Dons}{\Delta^n_{\otimes}(s_1, \dots, s_n)}
\newcommand{\Dont}{\Delta^n_{\otimes}(t_1, \dots, t_n)}
\newcommand{\Doon}{\Delta^n_{\otimes_1}(r_1, \dots, r_n)}
\newcommand{\Dotn}{\Delta^n_{\otimes_2}(r_1, \dots, r_n)}
\newcommand{\coeq}{\text{coeq} \,}
\newcommand{\eq}{\text{eq} \, }
\newcommand{\Lan}{\operatorname{Lan}}
\newcommand{\sab}{\mathbf{sAb}}
\newcommand{\Ch}{\mathbf{Ch}}
\newcommand{\sabv}{\mathbf{sAb}^{\Vop}}
\newcommand{\Ab}{\mathbf{Ab}}
\newcommand{\scpx}{\mathbf{SCpx}}
\newcommand{\sing}{\operatorname{Sing}}
\newcommand{\vr}{\operatorname{VR}}
\newcommand{\oN}{\overline{N}}
\newcommand{\loc}{\operatorname{Loc}}
\newcommand{\mch}{\mathcal{H}}
\newcommand{\mbc}{\mathbf{c}}
\newcommand{\mbC}{\mathbf{C}}

\makeatletter
\newcommand{\neutralize}[1]{\expandafter\let\csname c@#1\endcsname\count@}
\makeatother

\newenvironment{thmbis}[1]
  {\renewcommand{\thetheorem}{\ref{#1}$'$}%
   \neutralize{theorem}\phantomsection
   \begin{theorem}}
  {\end{theorem}}

\title{Quantales, persistence, and magnitude homology}
\author{Simon Cho}

\begin{document}
\maketitle

\begin{abstract}

We construct a nerve functor parametrized by a choice of quantale, exhibiting both the Vietoris-Rips complex and the magnitude nerve as instances of this nerve for different choices of monoidal structure on $\mathbb{R}$. Furthermore, the difference between how persistent homology processes the Vietoris-Rips complex and how magnitude homology processes the magnitude nerve is cast as a choice of whether or not to ``localize'' the corresponding nerves along $\mathbb{R}$ in a precise sense. Lastly, we mention some application-oriented observations naturally suggested by the perspective mentioned above.

\end{abstract}

\section{Introduction}

In recent years the idea of studying a metric space by turning it into the data of an $\mathbb{R}$-indexed sequence of simplicial complexes and then applying homology pointwise along $\mathbb{R}$ has manifested, completely independently, in two different contexts: persistent homology and magnitude homology. The former, described in greater detail in \cite{ghrist}, is an effective tool in situations involving potentially noisy data sets where one seeks ``holes'' in the data, an instructive example of which is a point cloud (in e.g. the Euclidean plane) arranged in the rough outline of a circle. A human observer would immediately be able to identify such a space as ``roughly a circle'', and persistent homology is one answer to the need for a concise mathematical feature of the space which captures this observation. On the other hand, magnitude homology is an algebro-topological generalization of the notion of magnitude \cite{magnothom}, and is used e.g. to test for certain kinds of convexity in a metric space \cite{lsmaghom}.

While both persistent and magnitude homology follow the same general philosophy of applying homology in a metric analysis setting, the features they capture and thus the specific applications they are utilized towards are quite different. This is due to important differences in the way the machinery of persistent/magnitude homology processes each metric space into a sequence of simplicial complexes, which we address in this paper. To our knowledge, the first (and, until now, only) paper to explicitly study the relationship between persistent and magnitude homology is \cite{otter}, which casts the difference as a kind of ``blurring''. We offer an alternative perspective through which we exhibit persistent and magnitude homology as instances of the same general framework, where one gets persistent or magnitude homology depending on different choices of (1) monoidal structure on $\mathbb{R}$ (which is neglected in \cite{otter}) and (2) localization of the sequence of simplicial complexes (which is the ``blurring'' mentioned in \cite{otter}).

First in Section \ref{sec:setup} we set up the basics of the framework, by recalling the notion of a quantale and developing the perspective that metric spaces are just graphs enriched over quantales, after Lawvere's famous observation in \cite{lawveremetric}. In Section \ref{sec:simplicial} we examine the Vietoris-Rips complex and the ``enriched nerve'' of \cite{otter} from this perspective, and describe precisely in what sense they are simply occurrences of the same construction but for different choices of a parameter. Specifically, we take a (very simplified) abstract homotopy theoretic approach to studying quantale-enriched graphs, and show that there is a uniform way in which each choice of monoidal structure gives rise to a nerve functor; the Vietoris-Rips complex and the enriched nerve are then the nerve functors corresponding to distinct monoidal structures on $\mathbb{R}$, each of which may be specified by a single (extended) real number $p \in [1, \infty]$.

The second essential difference between persistent and magnitude homology is that whereas persistent homology retains the data of smaller-scale simplices at every scale $r \in \mathbb{R}$ (thus the term ``persistent''), magnitude homology deliberately forgets this data; the only simplices which survive at scale $r$ are those of size exactly $r$. This forgetting is referred to as ``blurring'' in \cite{otter}, while we prefer to think of it as a sort of localization along $\mathbb{R}$. This is developed in Section \ref{sec:persmaghom} where we ultimately show in Theorem \ref{thm:persvsmagprime} that persistent homology and magnitude homology correspond to two different ways to decompose a natural transformation from ``global $\ell^1$ homology'' to ``local $\ell^{\infty}$ homology''.

As an application of our ``quantalic'' perspective, in Section \ref{sec:apps} we show a few results that hold in settings where one may apply persistent/magnitude homology. While these are not direct corollaries of the statements proven in the preceding sections, they are answers to questions naturally and strongly suggested by the perspective developed in this paper. More specifically, exploring the effects of different choices of quantale leads us to the following results, stated more precisely in Section \ref{sec:apps}:
\begin{enumerate}[label=$\circ$, ref=$\circ$]

\item Persistent homology is an indicator of the failure of a metric space to be an ultrametric space.

\item ``Approximate magnitude homology'' is an indicator of the existence of ``approximately collinear'' points.

\item In the context of automata theory, given an automaton with inputs for which we have a good notion of cost (of enacting each input), magnitude homology is generated in degree $1$ by pairs of states for which there is an ``indecomposable'' transition which is cheaper than any composite transition between them.

\end{enumerate}
These are just a sampling of the results one is naturally led towards from the viewpoint presented in this paper, and the author would be interested in - and would encourage those in the persistent/magnitude homology communities to explore - further and deeper applications of this perspective.

The author is grateful to John Baez, Andreas Blass, Alexander Campbell, Justin Curry, Brendan Fong, Robert Ghrist, Simon Henry, Dirk Hoffmann, Chris Kapulkin, Emily Riehl, David Spivak, and Henry Towsner for helpful conversations.

\section{Setup}
\label{sec:setup}

This section recalls well-established parts of category theory that we will use throughout the rest of this paper; standard references are \cite{rosenthal} for quantales, \cite{kelly} for enriched categories, and \cite{lawveremetric} for the specific sense of enriched categories that we work with.

In a few places (Remark \ref{rmk:ordering} and Example \ref{ex:qs}) we will establish notational/naming conventions that differ from some more standard ones, so as to be more obviously suggestive of the ideas in this paper.

\begin{definition}
\label{def:frame}

Let $\mcv$ be a category. $\mcv$ is a \emph{frame} when it satisfies the following:

	\begin{enumerate}

	\item $\mcv$ is a complete lattice, i.e. $\mcv$ is a poset with all (small) joins and meets.

	\item Joins commute with finite meets: for all $r, s_i \in \mcv$, we have
	\begin{equation*}
	r \wedge \left ( \bigvee\limits_i s_i \right ) = \bigvee\limits_i \left ( r \wedge s_i \right )
	\end{equation*}

	\end{enumerate}

\end{definition}

We will think of frames as posets of truth values; we will elaborate on this later. For now, we note that frames are a particular case of the following more general notion of ``poset of truth values''.

\begin{definition}
\label{def:quantale}

Let $\mcv$ be a complete lattice.

\begin{enumerate}

\item A \emph{semicartesian monoidal structure} on $\mcv$ is the data of a monoidal structure $\otimes$ on $\mcv$ for which the unit is also the terminal object of $\mcv$. We write $(\mcv, \otimes)$ to refer to $\mcv$ with a specified such structure.

\item $(\mcv, \otimes)$ is an \emph{affine quantale} when joins in $\mcv$ commute with $\otimes$: for all $r, s_i \in \mcv$, we have
\begin{equation*}
r \otimes \left ( \bigvee\limits_i s_i \right ) = \bigvee\limits_i \left ( r \otimes s_i \right )
\end{equation*}

\end{enumerate}

\end{definition}

For tidiness we will henceforth refer to affine quantales as simply \emph{quantales}, and to their monoidal structures as \emph{quantale structures}.

\begin{remark}
\label{rmk:ordering}

For peculiar technical reasons we will frequently have occasion to speak of the opposite ordering on a given quantale. Therefore we adopt the following convention, which we will follow throughout the rest of this paper.

Given a quantale $(\mcv, \otimes)$ and any $r, s \in \mcv$, we write $r \leq s$ iff $s \rightarrow r$ in $\mcv$. That is, $\leq$ is the partial order on $\Vop$. Moreover, given any $r_i, r, s \in \mcv$, we write
\begin{align*}
\sup\limits_i r_i &= \bigwedge\limits_i r_i & \max (r, s) &= r \wedge s \\
\inf\limits_i r_i &= \bigvee\limits_i r_i & \min (r, s) &= r \vee s
\end{align*}
where the $\bigwedge, \bigvee$ above are taken in $\mcv$. That is, ``$\sup$'' and ``$\inf$'' of objects of $\mcv$ refer respectively to $\bigwedge$ and $\bigvee$ in $\mcv$, or equivalently, respectively to joins and meets in $\Vop$.

Similarly, we will refer to the \emph{terminal} object of $\mcv$ (thus the initial object of $\Vop$) as $0$ and the \emph{initial} object of $\mcv$ (thus the terminal object of $\Vop$) as $\infty$, for reasons that will become clear shortly.

\end{remark}

\begin{example}
\label{ex:qs}
\item
\begin{enumerate}

\item Every frame has a default quantale structure provided by the meet ($\max$) operation.

\item Let $\mathbb{R}$ denote the poset category whose objects are the \emph{extended nonnegative} real numbers $0 \leq r \leq \infty$, with the reverse of the usual ordering on the real numbers, i.e. $s \rightarrow r$ in $\mathbb{R}$ iff $r \leq s$ as real numbers. This agrees with the perspective that the objects of $\mathbb{R}$ are truth values for the equality (i.e. distance) predicate, so that a distance of $0$ means that equality takes the top truth value.

This is the motivation for the convention just adopted in Remark \ref{rmk:ordering}, whereby the usual ordering of the real numbers agrees with the ordering on $\Rop$ (so there is no confusion in referring to either/both simultaneously by $\leq$), $0$ is initial in $\Rop$, and $\infty$ is terminal in $\Rop$.

Certainly $\mathbb{R}$ is a frame. For any real $1 \leq p \leq \infty$, we also have that $(\mathbb{R} , +_p )$ is a quantale, where for nonnegative reals $r, s \in \mathbb{R}$ we define
\begin{equation*}
r +_p s = (r^p + s^p)^{\frac{1}{p}}
\end{equation*}
with the operations on the right hand side being the usual addition and exponentiation of real numbers, and $r +_{\infty} s = \max (r, s)$. Clearly $+_1$ is just the usual addition, so we denote it as $+$.

Note that for $p < q$ we have that $r +_p s \geq r +_q s$ for all $r, s \in \mathbb{R}$, with equality holding iff at least one of $r,s$ is $0$. As we shall see in more detail, this induces an ordering on the $p$-parametrized (for $p \in [1, \infty]$) variants of ``metric space''. In particular, satisfaction of the triangle inequality for $+_p$ for any choice $p \in [1,\infty]$ implies satisfaction of the usual triangle inequality (i.e. for $+$). \label{quantales}


\item $\mathbbm{2} = \{ \, (\top = ) \, 0 \, \longleftarrow \, \infty \, (= \bot) \, \}$ is the two-element Boolean algebra, which is clearly a frame. This naming/ordering of the objects of $\mathbbm{2}$ makes sense from the perspective of the above discussion of the ordering on $\mathbb{R}$.

We have the evident frame inclusion $\mathbbm{2} \hookrightarrow \mathbb{R}$ given by $0 \mapsto 0$ and $\infty \mapsto \infty$, which is also a quantale inclusion $(\mathbbm{2}, \max) \hookrightarrow (\mathbb{R}, +_p)$ for any $p \in [1, \infty]$.

\end{enumerate}

\end{example}

We may now speak of ``graphs enriched over frames'':

\begin{definition}
\label{def:vgph}

Let $\mcv$ be a frame. A \emph{$\mcv$-graph} $X$ is specified by the following data:

\begin{enumerate}

\item A set, which we also write as $X$, of \emph{objects} or \emph{vertices}

\item For each ordered pair of objects $a, b \in X$, a specified $X(a,b) \in \mcv$ such that $X(a, a)=0$ for all $a \in X$.

\end{enumerate}

\end{definition}

We think of a $\mcv$-graph as a sort of generalized metric space with distances valued in $\mcv$, where ``generalized'' refers to the fact that the $\mcv$-labeling of edges in $\mcv$-graphs need not be symmetric nor satisfy any form of triangle inequality. Indeed, if $\mcv = \mathbb{R}$, these are known in e.g. \cite{cech} as extended pseudo-quasi-semi-metric spaces; for reasons of economy we simply call them $\mathbb{R}$-graphs for now.

If in addition we have a specified quantale structure on $\mcv$ making it into a quantale $(\mcv, \otimes)$, then we may ask whether or not a given $\mcv$-graph $X$ is actually a $(\mcv, \otimes)$-category, whose definition we now recall:

\begin{definition}
\label{def:vocat}

Let $(\mcv, \otimes)$ be a quantale, and let $X$ be a $\mcv$-graph. Then $X$ is a \emph{$(\mcv, \otimes)$-category} when, for all ordered triples $a, b, c \in X$, we have that
\begin{equation*}
X(a,c) \leq X(a,b) \otimes X(b,c).
\end{equation*}

\end{definition}

Thus a $(\mcv,\otimes)$-category is just a $\mcv$-graph that satisfies the triangle inequality in the sense provided by $\otimes$. If we take $(\mcv, \otimes)$ to be $(\mathbb{R}, +)$, then a $(\mathbb{R},+)$-category is known in e.g. \cite{cech} as a quasi-pseudo-metric space, or increasingly commonly as Lawvere metric spaces (after \cite{lawveremetric}). Again for reasons of economy we will shortly establish different terminology.

For now, we note that given a frame $\mcv$ (resp. quantale $(\mcv,\otimes)$) the $\mcv$-graphs (resp. $(\mcv,\otimes)$-categories) naturally assemble into a category where the morphisms are maps on vertices which ``weakly decrease distance'':

\begin{definition}
\label{def:vcats}
\item

\begin{enumerate}

\item Let $\mcv$ be a frame. The \emph{category of $\mcv$-graphs}, which we write as \emph{$\vgph$}, has as its objects the $\mcv$-graphs.

A morphism $f: X \rightarrow Y$ from a $\mcv$-graph $X$ to a $\mcv$-graph $Y$ is given by a set function $f: X \rightarrow Y$ from the set of vertices of $X$ to the set of vertices of $Y$ such that for all $a, b \in X$ we have
\begin{equation*}
X(a, b) \geq Y(fa, fb).
\end{equation*}

\item Let $(\mcv, \otimes)$ be a quantale. The \emph{category of $(\mcv,\otimes)$-categories}, which we write as \emph{$\vocat$}, has as its objects the $(\mcv,\otimes)$-categories.

A morphism $f: X \rightarrow Y$ of $(\mcv,\otimes)$-categories is given by a morphism of the underlying $\mcv$-graphs.

\end{enumerate}

\end{definition}

As Example \ref{ex:qs} might suggest, our main examples arise when $\mcv = \mathbb{R}$, which we will refer to as metric spaces. As there are many different conventions and preferences in the literature as to precisely which flavor of metric spaces a given author is referring to, we state the following to prevent confusion in the rest of the paper.

\begin{definition}
\label{def:mets}

By the term \emph{metric space} we refer to an $\mathbb{R}$-graph.

\begin{enumerate}

\item Recall that for each (extended) real number $1 \leq p \leq \infty$ there is a quantale $(\mathbb{R}, +_p)$. By the term \emph{$\ell^p$ metric space} we refer to an $(\mathbb{R}, +_p)$-category.

	\begin{enumerate}[label=$\circ$,ref=$\circ$]

	\item When $p = \infty$ so that $+_p = \max$, we also refer to $\ell^{\infty}$ metric spaces as \emph{ultrametric spaces}.

	\end{enumerate}

\item We use the prefix \emph{symmetric} for a metric space $X$ to indicate that for all $a,b \in X$, $X(a,b) = X(b,a)$.

\item We use the prefix \emph{strict} for a (possibly symmetric) metric space $X$ to indicate that it has the property that for all $a,b \in X$, $X(a,b) = 0$ implies $a = b$.

\end{enumerate}

\end{definition}

Since we will have occasion to make the distinction, let us use the term \emph{honest metric space} to refer to a metric space in the classical sense, i.e. a space whose metric is ``positive definite'', symmetric, and satisfies the triangle inequality. In terms of our new terminology, honest metric spaces are exactly the strict symmetric $\ell^1$ metric spaces.

There is an adjunction (as noted e.g. in \cite{lawveremetric})
\begin{tikzcd}
\vgph
\arrow[shift left=.5em]{r}[name=F]{\Fr}
		& \vocat
		\arrow[shift left=.5em]{l}[name=U]{U}
		\arrow[from=F, to=U, phantom]{}[rotate=90]{\vdash}
\end{tikzcd}
where $U$ is the evident forgetful functor and $\Fr$ is the functor which takes a $\mcv$-graph $X$ and returns its \emph{free $(\mcv,\otimes)$-category} $\Fr X$, which is constructed as follows:
\begin{enumerate}

\item The vertices of $\Fr X$ are the vertices of $X$.

\item For each ordered pair of vertices $a,b \in X$,
\begin{equation*}
\Fr X(a,b) = \inf\limits_{\substack{n \, \in \, \mathbb{N} \\ x_0, \dots, x_n}} \left ( \bigotimes\limits_{1 \leq i \leq n} X(x_{i-1}, x_i) \right )
\end{equation*}
where the infimum is taken over all possible finite sequences $x_0, \dots, x_n$ of vertices satisfying $x_0 = a$ and $x_n = b$.

\end{enumerate}
One can easily verify that this assignment on objects of $\vgph$ extends uniquely to a functor $\Fr: \vgph \rightarrow \vocat$. Moreover, we have that $U$ embeds $\vocat$ fully and faithfully into $\vgph$ so that $\vocat$ is a reflective subcategory with $\Fr$ being the reflection.

The following shows that $\vgph$ is bicomplete. Note that there is an obvious forgetful functor from $\vgph$ to the category $\Set$ of sets which sends each $\mcv$-graph to its set of vertices; therefore when we ``take the (co)limit in $\Set$'' of a diagram in $\vgph$, we mean that we apply the forgetful functor to the diagram first, and then take the (co)limit in $\Set$.

\begin{lemma}
\label{lem:vgphlims}

Let $\mcv$ be a frame.

\begin{enumerate}

\item Given a set $I$ and $X_i \in \vgph$ for each $i \in I$, \label{indexed}

	\begin{enumerate}

	\item The coproduct $\coprod\limits_i X_i$ exists in $\vgph$ and has the following explicit description:

		\begin{enumerate}

		\item The set of vertices of $\coprod\limits_i X_i$ is the coproduct of the $X_i$ in $\Set$;

		\item For $a, b \in \coprod\limits_i X_i$,
		\begin{equation*}
		\coprod\limits_i X_i (a,b) = 
		\left \{ \begin{array}{cc}
		X_k (a,b) & \text{if } a,b \in X_{k} \text{ for some } k \in I \\[0.25em]
		\infty & \text{otherwise}
		\end{array} \right .
		\end{equation*}

		\end{enumerate}

	\item The product $\prod\limits_i X_i$ exists in $\vgph$ and has the following explicit description:

		\begin{enumerate}

		\item The set of vertices of $\prod\limits_i X_i$ is the product of the $X_i$ in $\Set$, so that for any $a \in \prod\limits_i X_i$, we have $a = (a_i \mid a_i \in X_i)$;

		\item With notation as above, given any $a, b \in \prod\limits_i X_i$,
		\begin{equation*}
		\prod\limits_i X_i (a,b) = \sup\limits_i X_i (a_i, b_i)
		\end{equation*}

		\end{enumerate}
	\label{prods}
	\end{enumerate}

\item Given a pair of parallel morphisms
\begin{tikzcd}
X
\arrow[shift left]{r}{f}
\arrow[shift right]{r}[swap]{g}
	& Y
\end{tikzcd}
in $\vgph$, \label{span}

	\begin{enumerate}

	\item Their coequalizer $\coeq (f,g)$ exists in $\vgph$ and has the following explicit description:

		\begin{enumerate}

		\item The set of vertices of $\coeq (f,g)$ is the coequalizer of
		\begin{tikzcd}
		X
		\arrow[shift left]{r}{f}
		\arrow[shift right]{r}[swap]{g}
			& Y
		\end{tikzcd}
		in $\Set$; denote this set by $Z$. Denote by $\pi: Y \rightarrow Z$ the canonical quotient map in $\Set$.

		\item For any $a,b \in \coeq(f,g)$,
		\begin{equation*}
		\coeq(f,g) (a,b) = \inf\limits_{\substack{c,d \, \in \, Y \\ \pi(c) = a \\ \pi (d) = b}} Y(c,d)
		\end{equation*}

		\end{enumerate}

	\item Their equalizer $\eq (f,g)$ exists in $\vgph$ and has the following explicit description:

		\begin{enumerate}

		\item The set of vertices of $\eq (f,g)$ is the equalizer of
		\begin{tikzcd}
		X
		\arrow[shift left]{r}{f}
		\arrow[shift right]{r}[swap]{g}
			& Y
		\end{tikzcd}
		in $\Set$; denote this set by $A$. Denote by $i: A \rightarrow X$ the canonical inclusion in $\Set$.

		\item For any $a, b \in \eq (f,g)$,
		\begin{equation*}
		\eq (f,g) (a,b) = X(ia, ib)
		\end{equation*}

		\end{enumerate}
	\label{eqs}
	\end{enumerate}

\end{enumerate}

\end{lemma}

\begin{proof}

With the above descriptions in hand, verification of the requisite universal properties is entirely straightforward.

\end{proof}

We note that if there is a quantale structure $\otimes$ on $\mcv$, then the constructions provided by Lemma \ref{lem:vgphlims}(\ref{prods}),(\ref{eqs}) work perfectly well, without modification, to provide small limits in $\vocat$. That is, $\vocat \subseteq \vgph$ is closed under taking limits in $\vgph$, and since this inclusion is a full embedding, taking limits in $\vgph$ agrees with taking limits in $\vocat$.

Moreover, since $\Fr: \vgph \rightarrow \vocat$ is a left adjoint, in fact a reflection, we also have that $\vocat$ has small colimits, and that these colimits are computed by taking the colimit in $\vgph$ and then applying $\Fr$.

We now consider the case of two distinct quantales $(\mcv, \otimes_1)$ and $(\mcv, \otimes_2)$ sharing the same underlying frame $\mcv$; Example \ref{ex:qs} (\ref{quantales}) provides many such examples all sharing the same underlying frame $\mathbb{R}$, the main case of interest in this paper. If $\otimes_1 \geq \otimes_2$ in the sense that $r \otimes_1 s \geq r \otimes_2 s$ for all $r, s \in \mcv$, then every $(\mcv, \otimes_2)$-category is automatically a $(\mcv, \otimes_1)$-category, so that we get a full embedding, indeed a forgetful functor
\begin{equation*}
U^1_2 : \vocatt \rightarrow \vocato
\end{equation*}
with left adjoint
\begin{equation*}
\Fr^1_2: \vocato \rightarrow \vocatt
\end{equation*}
which takes a $(\mcv, \otimes_1)$-category $X$ and modifies the distances according to the formula
\begin{equation*}
\left ( \Fr^1_2 X \right )(a,b) = \inf\limits_{\substack{n \, \in \, \mathbb{N} \\ x_0, \dots, x_n}} \left ( \operatorname*{\bigotimes\nolimits_2}\limits_{1 \leq i \leq n} X(x_{i-1}, x_i) \right )
\end{equation*}
whose easy verification of functoriality we leave to the reader.

We summarize the preceding discussion as follows:

\begin{proposition}
\label{prop:adjoints}

Let $(\mcv, \otimes_1)$ and $(\mcv, \otimes_2)$ be two quantales with the same underlying frame $\mcv$, with $\otimes_1 \geq \otimes_2$. Then the preceding constructions yield the following commutative diagram of adjunctions:
\begin{equation*}
\begin{tikzcd}[row sep = 8em, column sep = 5em]
	& \vgph
	\arrow[shift right = .5em]{dl}[swap, name=f1]{\Fr_1}
	\arrow[shift left = .5em]{dr}[name=f2]{\Fr_2}
		& \\
\vocato
\arrow[shift right = .5em]{ur}[swap, name=u1]{U_1}
\arrow[from=f1, to=u1, phantom]{}[rotate=135, pos=0.575]{\vdash}
\arrow[shift left = .5em]{rr}[name=f12]{\Fr^1_2}
	&
		& \vocatt
		\arrow[shift left = .5em]{ul}[name=u2]{U_2}
		\arrow[from=f2, to=u2, phantom]{}[rotate=45, pos=0.575]{\vdash}
		\arrow[shift left = .5em]{ll}[name=u21]{U^1_2}
		\arrow[from=f12, to=u21, phantom]{}[rotate=90]{\vdash}
\end{tikzcd}
\end{equation*}

\end{proposition}

\section{Simplicial constructions}
\label{sec:simplicial}

For each $n \in \mathbb{N}$ and for each $(r_1, \dots, r_n) \in \mcv^n$, we can define an object $\Gn \in \vgph$ by declaring that it has $n+1$ vertices $x_0, \dots, x_n$ and that
\begin{equation*}
\Gn (x_i, x_j) =
\left \{ \begin{array}{cc}
0 & \text{if } i = j \\
r_j & \text{if } i = j-1 \\[0.25em]
\infty & \text{otherwise}
\end{array}\right .
\end{equation*}
If we are moreover given a quantale structure on $\mcv$ making it into a quantale $(\mcv,\otimes)$, then we define an object $\Don = \Fr \Gn \in \vocat$. When clear from context we will write $\Don$ to also refer to $U \Don = U \Fr \Gn \in \vgph$, which is only a slight abuse of notation since $U$ is a full embedding.

There is evidently a canonical map $\Gn \rightarrow \Don$ in $\vgph$ given by the unit of the adjunction.

Explicitly we can describe the distances in $\Don$ by
\begin{equation*}
\Don (x_i, x_j) =
\left \{ \begin{array}{cc}
0 & \text{if } i = j \\[0.25em]
\bigotimes\limits_{i+1 \leq k \leq j} r_k & \text{if } i < j \\[1em]
\infty & \text{otherwise}
\end{array} \right .
\end{equation*}

With setup and notation as in Proposition \ref{prop:adjoints}, we deduce that
\begin{equation*}
\begin{tikzcd}[row sep = 8em, column sep = 2.5em]
	& \Gn
	\arrow[mapsto]{dl}[swap, name=f1]{\Fr_1}
	\arrow[mapsto]{dr}[name=f2]{\Fr_2}
		& \\
\Doon
\arrow[mapsto]{rr}[name=f12]{\Fr^1_2}
	&
		& \Dotn
\end{tikzcd}
\end{equation*}
from which we immediately get the following.

\begin{corollary}

Given $(\mcv, \otimes_1)$ and $(\mcv, \otimes_2)$ with $\otimes_1 \geq \otimes_2$, there is a canonical map
\begin{equation*}
\Doon \rightarrow \Dotn
\end{equation*}
for each $r_1, \dots, r_n \in \mcv$, which are components of the unit of the adjunction
\begin{equation*}
\begin{tikzcd}
\vocato
\arrow[shift left=.5em]{r}[name=F]{\Fr^1_2}
		& \vocatt
		\arrow[shift left=.5em]{l}[name=U]{U^1_2}
		\arrow[from=F, to=U, phantom]{}[rotate=90]{\vdash}
\end{tikzcd}
\end{equation*}

\end{corollary}

\subsection{$(\mcv,\otimes)$-categories via lifting conditions}

We briefly explore the relationship between $\mcv$-graphs and $(\mcv, \otimes)$-categories from the perspective motivated by the notation just introduced.

\begin{proposition}
\label{prop:lifting}

Given a quantale $(\mcv, \otimes)$ and $X \in \vgph$, we have that $X$ is a $(\mcv,\otimes)$-category if and only if, for each $n \in \mathbb{N}$ and for each $(r_1, \dots, r_n) \in \mcv^n$, any map $f: \Gn \rightarrow X$ has a unique lift/``filler'' as in
\begin{equation*}
\begin{tikzcd}[row sep = 4em, column sep = 4em]
\Gn
\arrow{r}{f}
\arrow{d}
	& X \\
\Don
\arrow[dotted]{ur}{\exists !}
	&
\end{tikzcd}
\end{equation*}

\end{proposition}

\begin{proof}

Assume that $X \in \vgph$ satisfies the lifting condition. Pick any three vertices $a,b,c \in X$, and let $r = X(a,b)$, $s=X(b,c)$. Then by hypothesis the evident map $\Gamma^2(r,s) \rightarrow X$ has a lift $\Delta^2_{\otimes} (r,s) \rightarrow X$. But then we must have that $X(a,c) \leq r \otimes s = X(a,b) \otimes X(b,c)$.

Now assume that $X \in \vgph$ is in fact a $(\mcv, \otimes)$-category, so that it is in the image of the forgetful functor $U$. Then any map $\Gn \rightarrow X$ must factor uniquely through $\Gn \rightarrow \Don$ since this last map is the unit of the adjunction.

\end{proof}

We may think of Proposition \ref{prop:lifting} above as a kind of Segal condition for $\mcv$-graphs, which casts the triangle inequality as a kind of path composition (although not a particularly sophisticated one, since all paths are unique in this context).

Following this perspective further, we can thus think of applying $U \Fr$ to a $\mcv$-graph $X$ as a sort of fibrant replacement, in analogy with the usual fibrant replacement (pick any one) of simplicial sets that takes a simplicial set and returns an approximation of the original with all of the path compositions added in.

\subsection{A nerve functor}
\label{subsec:nerve}

Fix a quantale $(\mcv, \otimes)$. For each $X \in \vgph$ and each $n \in \mathbb{N}$, there is a functor $\widetilde{X}_n: (\Vop)^n \rightarrow \Set$ given by $\widetilde{X}_n(r_1, \dots, r_n) = \vgph ( \Don, X )$ whose action on morphisms $(r_1, \dots, r_n) \rightarrow (s_1, \dots, s_n)$ of $(\Vop)^n$ is the precomposition
\begin{equation*}
\vgph ( \Don, X ) \xrightarrow{\sigma^{\ast}} \vgph ( \Dons, X )
\end{equation*}
by the evident map $\sigma: \Dons \rightarrow \Don$.

Also for each $n \in \mathbb{N}$ there is a functor $\bigotimes: (\Vop)^n \rightarrow \Vop$ given by $(r_1, \dots r_n) \mapsto \bigotimes\limits_i r_i$. Taking the left Kan extension along $\bigotimes$, we get a functor
\begin{equation*}
X_n = \Lan_\otimes \widetilde{X}^n : \Vop \rightarrow \Set
\end{equation*}
i.e. a presheaf $X_n$ on $\mcv$.

\begin{proposition}
\label{prop:presimp}
\item

\begin{enumerate}

\item $X_n(r)$ is the set of $(n+1)$-tuples $(x_0, \dots, x_n)$ of vertices of $X$ such that $\exists (r_1, \dots, r_n) \in \mcv^n$ for which $\bigotimes\limits_{1 \leq i \leq n} r_i \leq r$, and for each $0 \leq i \leq j \leq n$,
	\begin{equation*}
	X(x_i, x_j) \leq \bigotimes\limits_{i+1 \leq k \leq j} r_k
	\end{equation*}

\item The action of $X_n$ on $r \leq s$ is the evident inclusion $X_n(r) \rightarrow X_n(s)$.

\end{enumerate}

\end{proposition}

\begin{proof}

Unpacking the formula for the left Kan extension, for each $r \in \mcv$ we have the following description of $X_n(r)$:
\begin{equation*}
X_n(r) = \left . \left \{ f: \Don \rightarrow X \mid \bigotimes_i r_i \leq r \right \} \right / \sim
\end{equation*}
where $\sim$ is the equivalence relation generated by the reflexive binary relation $\sim^{\prime}$, where $(f: \Don \rightarrow X) \sim^{\prime} (g: \Dons \rightarrow X)$ when $r_i \leq s_i$ for each $1 \leq i \leq n$ so that there is an evident map $\sigma: \Dons \rightarrow \Don$, and $g = \sigma^{\ast} f$.

It is easy to see that $(f: \Don \rightarrow X) \sim (g: \Dons \rightarrow X)$ iff there is some $(t_1, \dots, t_n)$ with $t_i \leq \min (r_i, s_i)$ so that there are evident maps $\rho: \Don \rightarrow \Dont$, $\sigma: \Dons \rightarrow \Dont$, and some $h: \Dont \rightarrow X$ such that $\rho^{\ast} h = f$ and $\sigma^{\ast} h = g$.

Then given an equivalence class of maps $(f: \Don \rightarrow X)$, we know that all of them must span exactly the same tuple $(x_0, \dots, x_n)$ of vertices in $X$. Conversely, any two maps $(f: \Don \rightarrow X)$ and $(g: \Dons \rightarrow X)$ which span the same tuple of vertices in $X$, taking $t_i = \min(r_i, s_i)$ there is clearly a map $(h: \Dont \rightarrow X)$ witnessing the equivalence of $f$ and $g$.

This shows that $X_n(r)$ is in bijective correspondence with the set of $(n+1)$-tuples of vertices of $X$ for each of which there exists \emph{some} map $f: \Don \rightarrow X$ with $\bigotimes\limits_{1 \leq i \leq n} r_i$. But this is exactly the statement that there is some $(r_1, \dots, r_n) \in \mcv^n$ such that $\bigotimes\limits_{1 \leq i \leq n} r_i \leq r$ and such that for each $1 \leq i < j \leq n$, $X(x_i, x_j) \leq \bigotimes\limits_{i+1 \leq k \leq j} r_k$.

For the second part, the Kan extension formula already guarantees us that for $r \leq s$, the inclusion
\begin{equation*}
\{f: \Don \rightarrow X \mid \bigotimes\limits_i r_i \leq r \} \hookrightarrow \{f: \Don \rightarrow X \mid \bigotimes\limits_i r_i \leq s \}
\end{equation*}
canonically descends to a map $X_n(r) \rightarrow X_n(s)$ that gives the action of $X_n$ on $r \leq s$; what we have shown then also makes it clear that this resulting map is also an inclusion.

\end{proof}

We claim that all of these $X_n$ assemble to give a functor $X: \Delta^{\text{op}} \rightarrow \PSh(\mcv)$, i.e. an object $X \in \PSh(\mcv \times \Delta)$.

\begin{proposition}
\label{prop:simpobj}

Given $X \in \vgph$, the assignment
\begin{equation*}
[n] \mapsto X_n
\end{equation*}
extends to a functor
\begin{equation*}
N_{\otimes}(X): \Dop \rightarrow \PSh(\mcv).
\end{equation*}
\end{proposition}

Since $\PSh(\Delta, \PSh(\mcv)) \cong \PSh(\mcv \times \Delta) \cong \sPSh(\mcv)$, this means that for each $X \in \vgph$ we get (abusing notation) a simplicial presheaf $N_{\otimes}(X)$ on $\mcv$. Henceforth we will move freely between these perspectives on $N_{\otimes}(X)$; to eliminate any confusion, we will use the notation $N_{\otimes}(X)(r, [n])$ when speaking of $N_{\otimes}(X)$ as an object of $\PSh(\mcv \times \Delta)$, and $N_{\otimes}(X)(r)_{n}$ when speaking of $X$ as an object of $\sPSh(\mcv)$.

\begin{proof}[Proof of Proposition \ref{prop:simpobj}]

We will show that the assignment $(r, [n]) \mapsto X_n(r)$ for each $(r, [n]) \in \mcv \times \Delta$ defines a functor $N_{\otimes}(X): (\Vop \times \Dop) \rightarrow \Set$ such that for each $r \leq s$ the map $N_{\otimes}(X)(r \leq s): X_n(r) \rightarrow X_n(s)$ is the evident inclusion.

For each $(r, [n]) \in \mcv \times \Delta$ and an $(n+1)$-tuple $(x_0, \dots, x_n) \in X_n(r)$, the $i^{\text{th}}$ face map omits the $i^{\text{th}}$ vertex and returns the $n$-tuple $(x_0, \dots, \hat{x}_i, \dots, x_n) \in X_{n-1}(r)$, while the $i^{\text{th}}$ degeneracy map repeats the $i^{\text{th}}$ vertex and returns the $(n+2)$-tuple $(x_0, \dots, x_i, x_i, \dots, x_n) \in X_{n+1}(r)$. It is a straightforward verification that this (is well defined and) satisfies the requisite simplicial identities.

For each $r \leq s$ we just define the map $N_{\otimes}(X)(r \leq s): X_n(r) \rightarrow X_n(s)$ to be the evident inclusion; it is clear that this commutes with the face and degeneracy maps given above.

\end{proof}

From how we defined it, the construction of $N_{\otimes}(X)$ is evidently natural in $X \in \vgph$, so that we get a functor $N_{\otimes}: \vgph \rightarrow \sPSh(\mcv)$. We should emphasize that, as the notation suggests, $N_{\otimes}$ depends not only on the underlying frame $\mcv$ but also on the quantale structure.

We observe the following relationship between $N_{\otimes}$ and $N_{\max}$:

\begin{lemma}
\label{lem:nattrans}

Given any $(\mcv, \otimes)$ there is a natural inclusion
\begin{equation*}
\begin{tikzcd}[column sep = 6em]
N_{\otimes}
\arrow{r}{\eta}
	& N_{\max}
\end{tikzcd}
\end{equation*}
of $N_{\otimes}$ as a subfunctor of $N_{\max}$.

\end{lemma}

\begin{proof}

We want to show that there is a natural inclusion of
\begin{equation*}
N_{\otimes}(X)(r)_n = \left \{ (x_0, \dots, x_n) \in X^{n+1} \biggm| \begin{array}{l}
\exists \, r_1, \dots, r_n \text{ such that } \bigotimes\limits_{1 \leq i \leq n} r_i \leq r \\
\text{and } \forall \, i \leq j , \; \; X(x_i, x_j) \leq \bigotimes\limits_{i+1 \leq k \leq j} r_k
\end{array}        \right \}
\end{equation*}
into
\begin{equation*}
N_{\max}(X)(r)_n = \left \{ (x_0, \dots, x_n) \in X^{n+1} \biggm| \begin{array}{l}
\exists \, r_1, \dots, r_n \text{ such that } \max\limits_{1 \leq i \leq n} r_i \leq r \\
\text{and } \forall \, i \leq j , \; \; X(x_i, x_j) \leq \max\limits_{i+1 \leq k \leq j} r_k
\end{array}        \right \}
\end{equation*}

So given $(x_0, \dots, x_n) \in N_{\otimes}(X)(r)_n$, there are $r_1, \dots, r_n \in \mathbb{R}$ such that $\bigotimes\limits_{1 \leq i \leq n} r_i \leq r$ and such that for all $i \leq j$, $X(x_i, x_j) \leq \bigotimes\limits_{i+1 \leq k \leq j} r_k$, we can take $s_k = \bigotimes\limits_{1 \leq i \leq n} r_i$ uniformly for all $1 \leq k \leq n$. Then for this choice of $s_i$ we have that
\begin{equation*}
\max\limits_{1 \leq i \leq n} s_i = \bigotimes\limits_{1 \leq i \leq n} r_i \leq r
\end{equation*}
while for all $1 \leq i \leq j \leq n$,
\begin{equation*}
X(x_i, x_j) \leq \bigotimes\limits_{i+1 \leq k \leq j} r_k \leq \bigotimes\limits_{1 \leq k \leq n} r_k \leq s_1 = \max\limits_{i+1 \leq k \leq j} s_k
\end{equation*}
which exhibits $(x_0, \dots, x_n)$ as an element of $N_{\max}(X)(r)_n$.

The fact that the tuple $(x_0, \dots, x_n) \in X^{n+1}$ as an element of $N_{\otimes}(X)(r)_n$ gets sent to the \emph{same} tuple $(x_0, \dots, x_n)$ considered as an element of $N_{\max}(X)(r)_n$ guarantees naturality of this inclusion.

\end{proof}

\subsection{The nerve-realization paradigm}

The construction of $N_{\otimes}$ clarifies in which sense it may be considered a nerve functor: for $X \in \vgph$, $N_{\otimes}(X) \in \sPSh(\mcv)$ is the ``singular complex indexed by total length of the singular simplices''. Moreover, in the special case that we take the default quantale structure on a frame $\mcv$ by setting $(\mcv, \otimes) = (\mcv, \max)$, the corresponding functor $N_{\max}$ coincides with the usual formal notion of \emph{nerve} as part of the nerve-realization paradigm, which we now describe briefly. (The author first learned of this formal perspective on $N_{\max}$ from Emily Riehl.)

We define a functor $G_0: (\mcv \times \Delta) \rightarrow \vgph$ by the assignment
\begin{equation*}
(r, [n]) \mapsto \Delta^n_{\max}(r, \dots, r)
\end{equation*}
with the evident action on morphisms.

It is an easy observation that every object of $\vgph$ is canonically a colimit of a diagram of the objects $\Delta^n_{\max}(r, \dots, r)$ (it even suffices to restrict to those objects where $n \leq 1$); we say that the objects $\Delta^n_{\max}(r, \dots, r)$ are \emph{dense} in $\vgph$. We should also note that for a \emph{fixed} $r \in \mcv$, the full subcategory of $\vgph$ on the objects $\Delta^n_{\max}(r, \dots, r)$ is clearly isomorphic to the simplex category $\Delta$.
We may take the left Kan extension $G = \operatorname{Lan}_{\yon} G_0$ along the Yoneda embedding $\yon: (\mcv \times \Delta) \rightarrow \PSh(\mcv \times \Delta)$ as in the diagram
\begin{equation*}
\begin{tikzcd}
	& \vgph \\
\mcv \times \Delta
\arrow{ur}{G_0}
\arrow{r}[swap]{\yon}
	& \PSh(\mcv \times \Delta)
	\arrow{u}[swap]{G}
\end{tikzcd}
\end{equation*}
which preserves colimits and thus has a right adjoint $N: \vgph \rightarrow \PSh(\mcv \times \Delta) \cong \sPSh(\mcv)$ that is fully faithful for formal reasons. The functors $G$ and $N$ are respectively the ``realization'' and ``nerve'' functors of the nerve-realization paradigm as applied to our setting.

Just from definitions and the Yoneda lemma we must have that, for $X \in \vgph$ and for each $(r, [n]) \in \mcv \times \Delta$,
\begin{equation*}
N (X) (r, [n]) = \vgph \left ( G(r, [n]) , X \right ).
\end{equation*}
Now we may regard $N(X)$ as either an object of $\PSh(\mcv \times \Delta)$ or as an object of $\sPSh(\mcv)$; for each $r \in \mcv$ and each $[n] \in \Delta$, we have that $N(X)(r)_n$ - the set of $n$-simplices of the simplicial set $N(X)(r)$ - is given by $\vgph \left ( G(r, [n]) , X \right )$, where $G(r,[n]) = G_0(r,[n]) = \Delta^n_{\max}(r, \dots, r)$. Thus the set of $n$-simplices in $N(X)(r)$ is given by the set of maps (in $\vgph$) from $\Delta^n_{\max}(r, \dots, r)$ into $X$. This is exactly the set of $(n+1)$-tuples $(x_0, \dots, x_n)$ of vertices of $X$ such that $X(x_i, x_j) \leq r$ for each $0 \leq i \leq j \leq n$, with the $i^{\text{th}}$ face (resp. degeneracy) maps given by omitting (resp. repeating) the $i^{\text{th}}$ vertex. But then by Theorem \ref{thm:vrcomplex} (\ref{nmax}) we have that $N(X)(r)$ is exactly $N_{\max} (X)(r)$. Furthermore it is not hard to see that $N(X)$ and $N_{\max}(X)$ agree on morphisms $r \leq s$, and also that $N$ and $N_{\max}$ agree on maps $X \rightarrow Y$, so that $N = N_{\max}$. Thus $N_{\max}$ occurs naturally as an instance of the nerve-realization paradigm; the construction of $N_{\otimes}$ for other quantale structures $\otimes$ on the frame $\mcv$ is just  an extension of this framework.

\begin{remark}
\label{rmk:cosk}

We will not have occasion to use this, but we note that $N$ above actually embeds $\vgph$ as the full reflective subcategory of $\PSh(\mcv \times \Delta)$ on the pointwise $1$-coskeletal objects, with $G$ being the reflection.

\end{remark}

\subsection{The Vietoris-Rips complex}

In the special case that we take $(\mcv, \max)$ as our quantale for some frame $\mcv$, then $N_{\max}(X)$ has a simple description, as we show in Theorem \ref{thm:vrcomplex} below. If furthermore we take $\mcv = \mathbb{R}$, then this yields precisely the familiar Vietoris-Rips complex (whose construction is detailed in e.g. \cite{ghrist}), which we also show in Theorem \ref{thm:vrcomplex}, in part by appealing to results shown in \cite{omar}. We first recall some relevant definitions.

\begin{definition}
\label{def:cplx}

\item

\begin{enumerate}

\item A \emph{simplicial complex} $X$ is specified by the following data:

	\begin{enumerate}

	\item A set $V_X$ of \emph{vertices};

	\item A subset $S_X \subseteq \mathcal{P}^{\text{fin}}_{\neq \emptyset} (V_X)$ of the poset (ordered by inclusion) of nonempty finite subsets of $V_X$, such that
		\begin{enumerate}

		\item $S_X$ is downward closed as a subposet of $\mathcal{P}^{\text{fin}}_{\neq \emptyset} (V_X)$;

		\item $S_X$ contains all singletons.

		\end{enumerate}

	A set $\{x_0, \dots, x_n\} \in S_X$ of $n+1$ vertices is called an \emph{$n$-simplex} of $X$.

	\end{enumerate}

\item A morphism of simplicial complexes $X \rightarrow Y$ is a function $f: X \rightarrow Y$ that satisfies the condition that, for each $x \in S_X$, the image $f[x] \in \mathcal{P}^{\text{fin}}_{\neq \emptyset} (V_Y)$ is an element of $S_Y$.

\item We denote the resulting category of simplicial complexes by $\scpx$.

	\begin{enumerate}[label=$\circ$, ref=$\circ$]

	\item Denote by $\sing: \scpx \rightarrow \sSet$ the functor defined on objects (and extended in the obvious way to morphisms) by
	\begin{equation*}
	\sing(X)_n = \left \{ (x_0, \dots, x_n) \in (V_X)^{n+1} \mid \{x_0, \dots, x_n\} \in S_X \right \} \text{ for each } n \in \mathbb{N}
	\end{equation*}
	where the action of the $i^{\text{th}}$ face (resp. degeneracy) maps of $\Dop$ are given by omitting (resp. repeating) the $i^{\text{th}}$ vertex.

	\end{enumerate}

\item Denote by $\scpx^{\Rop}$ the category of functors from $\Rop$ to $\scpx$. Given an honest metric space $X$, its \emph{Vietoris-Rips complex} is the $\Rop$-indexed simplicial complex $\vr(X) \in \scpx^{\Rop}$ where $\vr(X)(r)$ is the simplicial complex given by

	\begin{enumerate}

	\item The vertices of $\vr(X)(r)$ are the points of $X$, on which the action of morphisms $r \leq s$ is the identity;

	\item The $n$-simplices of $\vr(X)(r)$ are the sets $\{x_0, \dots, x_n\}$ of $n+1$ distinct points of $X$ which are at pairwise distance $\leq r$.

	\end{enumerate}

\item Denoting by $\Met$ the category of honest metric spaces with morphisms the $1$-Lipschitz functions (equivalently, the full subcategory of $\rgph$ on the strict symmetric $\ell^1$ metric spaces), the Vietoris-Rips complex is a functor $\vr: \Met \rightarrow \scpx^{\Rop}$.

\end{enumerate}

\end{definition}

The point of the somewhat lengthy recalling of definitions above is the precise statement of the following.

\begin{theorem}
\label{thm:vrcomplex}

\item

\begin{enumerate}

\item Let $\mcv$ be a frame and $X \in \vgph$. For each $r \in \mcv$, each $n$-simplex in $N_{\max}(X)(r)$ exactly corresponds to an $(n+1)$-tuple $(x_0, \dots, x_n)$ of vertices of $X$ such that for each $0 \leq i \leq j \leq n$, $X(x_i, x_j) \leq r$.\label{nmax}

\item If $\mcv = \mathbb{R}$, then $N_{\max}$ yields the Vietoris-Rips complex in the following distinct yet related senses:

	\begin{enumerate}

	\item There is a (non-functorial) assignment to each honest metric space $X$ an $\ell^1$ metric space $X^\prime$ such that for each $r \in \mathbb{R}$ the geometric realization of $\vr(X)(r)$ is (homeomorphic to) the geometric realization of $N_{\max} (X^\prime)(r)$.\label{vrnonfunc}

	\item The following diagram of functors commutes:
	\begin{equation*}
	\begin{tikzcd}[row sep = 4em, column sep = 4em]
	\Met
	\arrow[hook]{r}
	\arrow{d}[swap]{\vr}
		& \rgph
		\arrow{d}{N_{\max}} \\
	\scpx^{\Rop}
	\arrow{r}{\sing}
		& \sPSh(\mathbb{R})
	\end{tikzcd}
	\end{equation*}
	and moreover for each $X \in \Met$, and for each $r \in \mathbb{R}$, the geometric realization of $\vr(X)(r)$ is homotopy equivalent \cite{omar} to the geometric realization of $\sing(\vr(X))(r)$.\label{vrfunc}

	\end{enumerate}

\end{enumerate}

\end{theorem}

\begin{proof}

(\ref{nmax}): We know from Proposition \ref{prop:presimp} that
\begin{equation*}
N_{\max}(X)(r)_n = \left \{
\begin{array}{c | c}
(x_0, \dots, x_n)
	& \begin{array}{c}
	\exists (r_1, \dots, r_n) \in \mcv^n \, \text{ such that } \max\limits_{1 \leq k \leq n} r_k \leq r, \text{ and } \\
	X(x_i, x_j) \leq \max\limits_{i+1 \leq k \leq j} r_k \, \text{ for each } 0 \leq i \leq j \leq n
	\end{array}
\end{array}\right \}
\end{equation*}
and our assertion is that this is the same as the set
\begin{equation*}
\left \{
\begin{array}{c | c}
(x_0, \dots, x_n)
	& X(x_i, x_j) \leq r \, \text{ for each } 0 \leq i \leq j \leq n
\end{array}\right \}.
\end{equation*}
Clearly the former is a subset of the latter, and the latter is a subset of the former since we may take $r_k = r$ for each $1 \leq k \leq n$.

(\ref{vrnonfunc}): This is a very slight adaptation of a well known construction, see e.g. \cite{omar}, \cite{jardine}. Given an honest metric space $X$ with metric $d_X$, let us put a total order $\preceq$ on the points of $X$. Let us define $X^\prime$ by declaring that its vertices are the points of $X$, and that
\begin{equation*}
X^\prime(x,y) = \left \{
\begin{array}{cc}
d_X(x,y)
	& \text{if } x \preceq y \\
\infty
	& \text{otherwise}
\end{array} \right .
\end{equation*}
The $n$-simplices of $\vr(X)(r)$ are the sets $\{x_0, \dots, x_n\}$ of distinct points of $X$ at pairwise distance $\leq r$. The faces of each such simplex are the $n$-element subsets.

On the other hand, given such a set $\{x_0, \dots, x_n\}$, by the result (\ref{nmax}) and the construction of $X^{\prime}$ there is exactly one $(n+1)$-tuple $(x_{i_0}, \dots, x_{i_n}) \in N_{\max}(X^{\prime})(r)_n$ with $k \mapsto i_k$ a permutation of the set $\{0, \dots, n\}$, namely the one where $x_{i_j} \preceq x_{i_k}$ iff $j \leq k$. Similarly given any subset of $n$ elements, there is exactly one $n$-tuple with those same elements contained in $N_{\max}(X^{\prime})(r)_{n-1}$, namely the one which omits the unique element not contained in the $n$-element subset; then one of the face maps takes $(x_{i_0}, \dots, x_{i_n})$ to this $(n-1)$-simplex. 

Thus for any \emph{set} $\{x_0, \dots, x_n\}$ of distinct points of $X^{\prime}$ there is at most a single $n$-simplex whose vertices consist of those points, in any order, in $N_{\max}(X^\prime)(r)$, for any $r \in \mathbb{R}$. Moreover $N_{\max}(X^{\prime})(r)$ contains such an $n$-simplex iff $\{x_0, \dots, x_n\}$ is an $n$-simplex of $\vr(X)(r)$. Finally, any $(x_0, \dots, x_n) \in N_{\max}(X^\prime)(r)_n$ which has at least one repeated element must be a degenerate simplex. These facts together make it clear that for each $r \in \mathbb{R}$ the geometric realizations of $\vr(X)(r)$ and $N_{\max}(X^\prime)(r)$ are the same.

(\ref{vrfunc}): Let $X$ be an honest (i.e. strict symmetric $\ell^1$) metric space. Then considering $X$ as an $\mathbb{R}$-graph, we have that $X(a,b) = X(b,a)$ for all $a,b \in X$. Thus for any given $r \in \mathbb{R}$ and set $\{x_0, \dots, x_n\}$ of points of $X$, $(x_0, \dots, x_n) \in N_{\max}(X)(r)_n$ iff $(x_{i_0}, \dots, x_{i_n}) \in N_{\max}(X)(r)_n$ for every permutation $k \mapsto i_k$ of $\{0, \dots, n\}$. From this and (\ref{nmax}) we conclude that $(x_0, \dots, x_n) \in N_{\max}(X)(r)_n$ iff the points in the set $\{x_0, \dots, x_n\}$ are at pairwise distance $\leq r$. However this last condition holds iff $\{x_0, \dots, x_n\} \in S_{\vr(X)(r)}$ by construction. Therefore
\begin{equation*}
N_{\max}(X)(r)_n = \{(x_0, \dots, x_n) \mid \{x_0, \dots, x_n\} \in S_{\vr(X)(r)}\} = \sing(\vr(X))(r)_n
\end{equation*}
with obvious agreement on the actions of morphisms of $\Rop$ (inclusion) and $\Dop$ (omission/repetition of vertices). This correspondence is also natural because it is specified completely in terms of tuples/sets of vertices, which is also the case for maps in $\Met$.

The fact that the geometric realizations of $\sing(\vr(X))(r)$ and $\vr(X)(r)$ are homotopy equivalent is a result of \cite{omar}.

\end{proof}

This is a central example for us, and as such we will have more to say about it later in this section and the next. For now let us remark that the construction given in (the proof of) Theorem \ref{thm:vrcomplex} is easier to work with since it is quite small; while applying $\sing$, being functorial, enjoys better categorical properties.

Recall that if $X \in \vgph$ is in fact a $(\mcv, \otimes)$-category, it satisfies the appropriate ``triangle inequality'' in the sense of $(\mcv, \otimes)$. In this case $N_{\otimes}$ and its description supplied by Proposition \ref{prop:presimp} simplifies to yield a construction that is mentioned in \cite{otter} under the name \emph{enriched nerve}:

\begin{corollary}
\label{cor:vcatnerve}

Let $X \in \vgph$. If in fact $X \in \vocat \subseteq \vgph$, then the description of each $N_{\otimes}(X)(r)_n$ simplifies further; it is the set of $(n+1)$-tuples $(x_0, \dots, x_n)$ of vertices of $X$ such that $\bigotimes\limits_{1 \leq i \leq n} X(x_{i-1}, x_i) \leq r$.

\end{corollary}

On the other hand, we emphasize that the nerve $N_{\otimes}$ is well-defined on all of $\vgph$, so that if $\mcv$ is a frame that supports distinct quantale structures $\otimes_1$ and $\otimes_2$, then $N_{\otimes_1}, N_{\otimes_2}: \vgph \rightarrow \sPSh(\mcv)$ are distinct nerve functors arising from the same general paradigm which, as we will show, may be fruitfully compared to each other.

Indeed, the key observation is that when $\mcv = \mathbb{R}$, then the nerve functors $N_{\max}$ and $N_+$ corresponding to the two different quantale structures ``$\max$'' and ``$+$'' on $\mathbb{R}$ respectively yield the Vietoris-Rips complex (by Theorem \ref{thm:vrcomplex}) and the enriched nerve used for magnitude homology e.g. in \cite{lsmaghom} and \cite{otter}. Just as the Vietoris-Rips complex is the restriction of $N_{\max}$ to the category of honest metric spaces, the enriched nerve in the context of \cite{otter} is the restriction of $N_+$ (more generally, of $N_{\otimes}$) to $\ell^1$ metric spaces (more generally, to $(\mcv, \otimes)$-categories).

In the interest of keeping the following material as clear as possible, as well as maintaining consistency with Definition \ref{def:mets}, let us use the term \emph{$\ell^p$ nerve} to refer to the nerve $N_{+_p}$ corresponding to the choice $(\mcv, \otimes) = (\mathbb{R}, +_p)$.

Therefore in this terminology, the Vietoris-Rips complex and the ``enriched nerve'' of magnitude homology are the restrictions of the $\ell^{\infty}$ nerve and $\ell^1$ nerve to $\ell^1$ metric spaces. We should also mention that there are good practical reasons for such restrictions, as the honest metric spaces tend to live in the sweet spot of having many examples relevant to applications while still being convenient/tractable to prove results about. The point is that moving to the more general setting allows us to see how the Vietoris-Rips complex and the ``enriched nerve'' arise from the same construction by changing the value of a parameter: they are both $\ell^p$ nerves, just for different values of $p$.

\section{Persistent and magnitude homology}
\label{sec:persmaghom}

We have seen how the Vietoris-Rips complex and the $\ell^1$ nerve are examples of the same construction ($\ell^p$ nerve). Even with this in mind, however, the ways they are further processed in order to yield the final products of persistent/magnitude homology are technically quite different. We will see that in this case also the difference is in some sense the tuning of a parameter (independent of the choice of $p$). We start by briefly recalling these different pipelines for processing the respective nerves.

\subsection{$\mathbb{R}$-indexed chain complexes}
\label{subsec:vrpipeline}

We will rely on established machinery for which a standard reference is \cite{gj}.

The Dold-Kan correspondence gives us an equivalence of categories
\begin{equation*}
\begin{tikzcd}
\sab
\arrow[shift left=0.35em]{r}{K}[swap]{\simeq}
	& \Ch
	\arrow[shift left=0.35em]{l}{D}
\end{tikzcd}
\end{equation*}
between the category $\sab$ of simplicial abelian groups and the category $\Ch$ of chain complexes concentrated in nonnegative degree, where $K$ is the normalized Moore complex functor and $D$ takes a chain complex and returns in some sense the ``free'' simplicial group on that complex.

We recall that given a simplicial abelian group $A \in \sab$, its normalized Moore complex $KA$ is given as follows:
\begin{enumerate}

\item $KA_n$ is $A([n])/D_n$, where $A([n])$ is the value of $A$ on $[n] \in \Delta$ and $D_n \subseteq A([n])$ is the subgroup of $A([n])$ generated by the degenerate elements.

\item The map $\partial: KA_n \rightarrow KA_{n-1}$ is the one induced by the alternating face map $\partial: A([n]) \rightarrow A([n-1])$ given by
\begin{equation*}
\sum\limits^n_{i=0} (-1)^{i} d_i
\end{equation*}
where $d_i: A([n]) \rightarrow A([n-1])$ is the $i^{\text{th}}$ face map given by the data of $A$ and the sum uses the addition provided by $A([n-1])$.

\end{enumerate}

That all of this is well-defined and works as intended is shown in e.g. \cite{gj}.

Noting that an equivalence of categories induces via pushforward an equivalence of the appropriate functor categories, for each quantale $(\mcv, \otimes)$ we have the composite
\begin{equation*}
\begin{tikzcd}
\vgph
\arrow{r}{N_{\otimes}}
	&\sPSh(\mcv)
	\arrow{r}{F_{\ast}}
		& \sabv
		\arrow[shift left=0.35em]{r}{K_{\ast}}[swap]{\simeq}
			& \Ch^{\Vop}
			\arrow[shift left=0.35em]{l}{D_{\ast}}
			\arrow{r}{H_{\bullet}}
				& \Ab^{\Vop}
\end{tikzcd}
\end{equation*}
where
\begin{enumerate}

\item $F: \sSet \rightarrow \sab$ the free abelian group functor;

\item $\sabv$ is the category of functors from $\Vop$ into $\sab$ (equivalently, functors from $(\Vop \times \Dop)$ into $\Ab$);

\item $\Ch^{\Vop}$ is the category of functors from $\Vop$ into $\Ch$;

\item $H_{\bullet}$ is the homology functor (applied pointwise on $\Vop$).

\end{enumerate}

Given a simplicial set $S$ we know that $H_{\bullet} K F (S)$ is the homology of $S$, so $H_{\bullet} K_{\ast} F_{\ast} N_{\otimes}$ simply takes the pointwise (on $\Vop$) homology of the nerve $N_{\otimes}(X)$ for any $X \in \sPSh(\mcv)$.

For future convenience, we establish the following terminology:

\begin{definition}
\label{def:lphom}

Let $(\mcv, \otimes)$ be given.

\begin{enumerate}

\item In the notation of the preceding, let us write $F_{\ast} N_{\otimes}$ as $\oN_{\otimes}$.

	\begin{enumerate}[label=$\circ$, ref=$\circ$]

	\item When $\mcv = \mathbb{R}$ so that for each $1 \leq p \leq \infty$ we have the $\ell^p$ nerve $N_{+_p}$, let us also refer to $\oN_{+_p}$ as the \emph{$\ell^p$ complex}.

	\end{enumerate}

\item Taking the equivalence of categories
	\begin{tikzcd}
	 \sabv
		\arrow[shift left=0.35em]{r}{K_{\ast}}[swap]{\simeq}
			& \Ch^{\Vop}
			\arrow[shift left=0.35em]{l}{D_{\ast}}
	\end{tikzcd} as established, we henceforth suppress mention of $K_{\ast}$ and $D_{\ast}$, and write $H_{\bullet}: \sab^{\Vop} \rightarrow \Ab^{\Vop}$ to actually mean the composite
	\begin{equation*}
	\sab^{\Vop} \xrightarrow{K_{\ast}} \Ch^{\Vop} \xrightarrow{H_{\bullet}} \Ab^{\Vop}.
	\end{equation*}

\item We may refer to $H_{\bullet}\oN_{\otimes}$ as \emph{$\otimes$-homology}.

	\begin{enumerate}[label=$\circ$,ref=$\circ$]

	\item When $\mcv = \mathbb{R}$, let us also refer to $+_p$-homology as \emph{$\ell^p$ homology}.

	\end{enumerate}

\end{enumerate}

\end{definition}

Since the $\ell^{\infty}$ nerve $N_{\max}$ is just the Vietoris-Rips complex in the sense made precise by Theorem \ref{thm:vrcomplex}, and since persistent homology is the pointwise-on-$\Rop$ homology of the Vietoris-Rips complex, we may say that \emph{persistent homology is $\ell^{\infty}$ homology} (cf. \cite{spivakfuzzy} for a related categorical perspective on persistent homology).

\subsection{Localizing the $\ell^p$ complex}
\label{subsec:loc}

Whereas persistent homology is the (pointwise) homology of the $\ell^\infty$ complex, we will see in this section that magnitude homology is the (pointwise) homology of an appropriately ``localized'' $\ell^1$ complex. Thus persistent homology and magnitude homology differ in two respects, namely the localization and the value of $p$ in the $\ell^p$ complex involved. The latter difference seems to have been overlooked in \cite{otter} when it is stated there that ``persistent homology is blurred magnitude homology''; we will account for this difference of ``blurring'' in this section as a kind of localization, but the point is that this ``blurring'' is only half of the difference between persistent and magnitude homology.

\begin{remark}
\label{rmk:otter}

The statement of \cite{otter} is still true if one accepts their characterization of persistent homology as homology of the $\ell^1$ complex, but the term ``persistent homology'', at least as applied to metric spaces, refers to homology of the $\ell^\infty$ (i.e. Vietoris-Rips) complex in the vast majority, if not all, of its occurrences in the literature (of which \cite{ghrist}, \cite{silvaghrist}, \cite{zcpershom} are a few well-known examples), and this is the sense of ``persistent homology'' to which we adhere in this paper.

On the other hand, we will commit our own etymological crime in this paper: our usage of the terms ``local'' and ``localization'' in what follows does \emph{not} refer to their usual notions, but rather a kind of localization \emph{along $\Rop$} for which we were unable to conjure satisfactory alternative terminology.

\end{remark}

We recall the pipeline of magnitude homology briefly. Whereas establishing its deeper theoretical properties is rather involved and is done in \cite{lsmaghom}, the actual explicit description of the objects is quite simple (and is again found e.g. in \cite{lsmaghom}):

\begin{definition}
\label{def:maghom}

Let $X$ be an $\ell^1$ metric space.

\begin{enumerate}

\item The \emph{magnitude nerve of $X$}, which we denote by $B(X) \in \sab^{\Rop}$, is the $\Rop$-indexed simplicial abelian group for which $B(X)(r)_n$ is the free abelian group on the set
\begin{equation*}
\left \{ (x_0, \dots, x_n) \mid \sum\limits_{1 \leq i \leq n} X(x_{i-1}, x_i) = r \right \}
\end{equation*}
with the evident action of face and degeneracy maps (any face whose lengths do not sum to $r$ is trivial in $B(X)(r)_{n-1}$). If $r < s$ is any nonidentity morphism of $\Rop$, then $B(X)(r) \rightarrow B(X)(s)$ is the zero map.

\item The \emph{magnitude homology of $X$} is the pointwise homology of $B(X)$.

\end{enumerate}

\end{definition}

Technically in \cite{lsmaghom}, \cite{otter} they regard $B(X)$ as being \emph{graded} on $\Rop$ instead of being a \emph{functor} on $\Rop$. It matters little in the end, since the action of every nonidentity morphism $r < s$ of $\Rop$ on $B(X)(r)$ is zero; the same will therefore hold true when we apply homology. A small advantage of our perspective is that we are able to directly compare $B(X)$ to the $\ell^1$ complex $\oN_{+}$.

Let $X$ be an $\ell^1$ metric space, so that by Corollary \ref{cor:vcatnerve} we have for each $r \in \Rop$ that
\begin{equation*}
N_+(X)(r)_n = \left \{ (x_0, \dots, x_n) \mid \sum\limits_{1 \leq i \leq n} X(x_{i-1}, x_i) \leq r \right \}
\end{equation*}
so that $\oN_+(X)(r)_n$ is the free abelian group on this set, with the action of the maps of $\Dop$ on $\oN_+(X)(r)$ induced by that on $N_+(X)(r)$. Clearly for every $r \leq s$ in $\Rop$ the action on $\oN_+(X)(r)_n$ is subgroup inclusion into $\oN_+(X)(s)_n$.

Note that for each $s \leq r$ the group $\oN_+ (X)(s)_n$ sits inside of $\oN_+ (X)(r)_n$ in the obvious way. Clearly we have that
\begin{equation*}
\oN_+(X)(r)_n \, / \, \bigcup\limits_{s < r} \oN_+(X) (s)_n \cong B(X)(r)_n. 
\end{equation*}

Since the actions of $\Rop$ and $\Dop$ on $\oN_+(X)$ commute (because e.g. $\oN_+(X)$ is equivalently a functor $(\Rop \times \Dop) \rightarrow \Ab$) we must in fact have 
\begin{equation*}
\oN_+(X)(r) \, / \, \bigcup\limits_{s < r} \oN_+(X) (s) \cong B(X)(r). 
\end{equation*}

\begin{remark}
\label{rmk:hepwil}

The above perspective is essentially the one mentioned in passing in Remark 44 of \cite{hepwil}, as part of a section detailing a simplicial perspective on the constructions in magnitude homology - they use the simplicial approach in their paper in order to prove the K\"{u}nneth theorem for magnitude homology.

\end{remark}

This suggests the following. To each $r \in \Rop$ assign some sieve $J_r$ (i.e. a downward closed subposet of $\Rop$) on $r$ in such a way that whenever $r \leq s$ we have $J_r \subseteq J_s$; this gives the data of a functor $J: \Rop \rightarrow \mathcal{S}(\Rop)$, where $\mathcal{S}(\Rop)$ is the poset of downward closed subposets of $\mathbb{R}$, ordered by inclusion. Given any $\Rop$-indexed simplicial abelian group $A \in \sab^{\Rop}$, we define another such $\loc_J(A) \in \sab^{\Rop}$ as
\begin{equation*}
\loc_J (A) (r) = A(r) \, / \, A(J_r)
\end{equation*}
where $A(J_r) \subseteq A(r)$ is the union of the images of the maps $A(s) \rightarrow A(r)$ for all $s \in J_r$. There is the minor issue of checking that this is well-defined, namely that
\begin{enumerate}

\item $A(J_r)$ is in fact a (simplicial) subgroup of $A(r)$, and that

\item The image of $A(J_r)$ under $A(r \leq s)$ is contained in $A(J_s)$ for $r \leq s$.

\end{enumerate}
If $a_1$ is in the image of $A(s_1)_n \rightarrow A(r)_n$ and $a_2$ is in the image of $A(s_2)_n \rightarrow A(r)_n$ for $s_1, s_2 \in J_r$ then $a_1 + a_2$ is certainly in the image of $A( \max(s_1, s_2) )_n \rightarrow A(r)_n$, verifying the first part. The second part is immediate by construction of $J$. Note that while this argument implicitly uses the total order on $\mathbb{R}$, we may repeat it for more general choices of frame $\mcv$ in place of $\mathbb{R}$ if we additionally require that for each $r \in \mcv$, the sieve $J_r$ is closed under taking $\max$, that is, $\max: (\Vop \times \Vop) \rightarrow \Vop$ restricts to $\max: (J_r \times J_r) \rightarrow J_r$ (this condition is satisfied automatically if $\mcv$ is totally ordered).

Let us call $\loc_J(A)$ the \emph{$J$-localization of $A$}. We say that $A \in \sab^{\Rop}$ is \emph{$J$-local} when $A \cong \loc_J(A)$, equivalently when for each $r \in \Rop$ and for each $s \in J_r$ the map $A(s) \rightarrow A(r)$ is the zero map. Then for each $J$ as above we get an adjunction
\begin{tikzcd}
\sab^{\Rop}
\arrow[shift left = 0.5em]{r}[name=L]{\loc_J}
	& \sab^{\Rop}_J
	\arrow[shift left = 0.5em]{l}[name=i]{\iota_J}
	\arrow[from=L, to=i, phantom]{}[rotate=90]{\vdash}
\end{tikzcd}
where $\sab^{\Rop}_J$ is the full subcategory of $\sab^{\Rop}$ on the $J$-local objects and $\iota_J$ is the inclusion; the requisite universal property is straightforward to check, or one may simply note that the property of being $J$-local must be preserved under taking limits since limits are taken pointwise.

\begin{remark}

In fact the same argument with colimits shows that $\iota_J$ is also a \emph{left} adjoint.

\end{remark}
If for each $r \in \Rop$ we let $J_r$ be the maximal nontrivial sieve, i.e. the set of all $s < r$, then we immediately see that $B(X) = \loc_J ( \oN_+(X))$. That is, the magnitude nerve is just the $J$-localization of the $\ell^1$ complex. We may think of this as the maximal localization we may perform that still fully preserves the metric data of $X$. In effect, this is the difference of ``blurring'' that is referred to in \cite{otter}. We may summarize the results of our discussion so far as follows:

\begin{theorem}
\label{thm:persvsmag}

Persistent homology and magnitude homology of metric spaces are the same construction with different (in fact opposite) choices of two parameters. More precisely, persistent homology is the homology of the \underline{un}localized $\ell^{\underline{\infty}}$ complex whereas magnitude homology of $X$ is the homology of the \underline{maximally} localized $\ell^{\underline{1}}$ complex.

\end{theorem}

Recalling that $\ell^p$ homology refers to homology of the $\ell^p$ complex, we may refer to it also by \emph{global $\ell^p$ homology} when we wish to emphasize that the $\ell^p$ complex has not been localized. Let us use the term \emph{local $\ell^p$ homology} to refer to homology of the maximally localized (i.e. $J$-localized, for each $J_r$ the maximal nontrivial sieve) $\ell^p$ complex.

We might then concisely say that persistent homology is global $\ell^\infty$ homology whereas magnitude homology is local $\ell^1$ homology. But even this statement does not fully capture the relationship between persistent homology and magnitude homology, on which we proceed to elaborate.

Let $\sPSh_m(\mathbb{R})$ denote the full subcategory of $\sPSh(\mathbb{R})$ on the objects $X$ for which each map $X(r) \rightarrow X(s)$ is a monomorphism (see e.g. \cite{jardine} for a discussion on the significance of such a subcategory), and similarly let $\sab^{\Rop}_m$ denote the full subcategory of $\sab^{\Rop}$ on the objects $A$ for which each map $A(r) \rightarrow A(s)$ is a monomorphism. Clearly the free-forgetful adjunction between $\sPSh(\mathbb{R})$ and $\sab^{\Rop}$ restricts to one between $\sPSh_m(\mathbb{R})$ and $\sab^{\Rop}_m$. Furthermore, for any $(\mcv, \otimes)$ the image of the corresponding nerve $N_{\otimes}$ is contained in $\sPSh_m(\mathbb{R})$ since for any $X \in \vgph$ and $r \leq s$ in $\Rop$, $N_{\otimes}(X)(r) \subseteq N_{\otimes}(X)(s)$.

In particular, for any $1 \leq p \leq \infty$ the $\ell^p$ complex of any $\mathbb{R}$-graph $X$ lives in $\sab^{\Rop}_m$ so that we can think of $\oN_{+_p}$ as a functor from $\rgph$ into $\sab^{\Rop}_m$. On the other hand, as we just saw we have that $\loc_J$ lands in $\sab^{\Rop}_J$. Thus we may regard global $\ell^p$ homology as the composition
\begin{equation*}
\begin{tikzcd}[column sep = 3em]
\rgph
\arrow{r}{N_{+_p}}
	& \sab^{\Rop}_m
	\arrow[hook]{r}{\iota_m}
		& \sab^{\Rop}
		\arrow{r}{H_{\bullet}}
			& \Ab^{\Rop}
\end{tikzcd}
\end{equation*}
and local $\ell^p$ homology as the composition
\begin{equation*}
\begin{tikzcd}[column sep = 4em]
\rgph
\arrow{r}{N_{+_p}}
	& \sab^{\Rop}_m
	\arrow{r}{\loc_J \circ \iota_m}
		&\sab^{\Rop}_J
		\arrow[hook]{r}{\iota_J}
			& \sab^{\Rop}
			\arrow{r}{H_{\bullet}}
				& \Ab^{\Rop}
\end{tikzcd}
\end{equation*}
where $\iota_m: \sab^{\Rop}_m \hookrightarrow \sab^{\Rop}$ is the inclusion.

Now by Lemma \ref{lem:nattrans} we have a natural inclusion $N_+ \hookrightarrow N_{\max}$, and by the discussion above we have a natural transformation 
\begin{tikzcd}[column sep = 6em]
\sab^{\Rop}_m
\arrow[shift left = 0.7em]{r}[name=m]{\iota_m}
\arrow[shift right = 0.7em]{r}[name=j, swap]{\iota_J \circ \loc_J \circ \iota_m}
	& \sab^{\Rop}
	\arrow[from=m, to=j, phantom]{}[rotate=90]{\Leftarrow}
\end{tikzcd}
given by the components of the unit of the adjunction
\begin{tikzcd}
\sab^{\Rop}
\arrow[shift left = 0.5em]{r}[name=L]{\loc_J}
	& \sab^{\Rop}_J
	\arrow[shift left = 0.5em]{l}[name=i]{\iota_J}
	\arrow[from=L, to=i, phantom]{}[rotate=90]{\vdash}
\end{tikzcd}
restricted to $\sab^{\Rop}_m \subseteq \sab^{\Rop}$. This gives the following deeper look at the relationship between persistent and magnitude homology.

Denote by $\eta: N_+ \rightarrow N_{\max}$ the natural transformation given by Lemma \ref{lem:nattrans} where $(\mcv, \otimes) = (\mathbb{R},+)$, i.e. the $2$-cell in
\begin{equation*}
\begin{tikzcd}[column sep = 8em]
\rgph
\arrow[bend left]{r}[name=m]{N_{\max}}
\arrow[bend right]{r}[name=j, swap]{N_+}
	& \sab^{\Rop}_m
	\arrow[from=m, to=j, phantom, shift right]{}[rotate=270]{\Longleftarrow}
	\arrow[from=m, to=j, phantom, shift left = 0.5em]{}{\eta}
\end{tikzcd}
\end{equation*}

Denote by $\varphi: \iota_m \rightarrow (\iota_J \circ \loc_J \circ \iota_m)$ the aforementioned natural transformation, i.e. the $2$-cell in
\begin{equation*}
\begin{tikzcd}[column sep = 3em, row sep = 1em]
\sab^{\Rop}_m
\arrow[bend left]{rr}[name=m]{\iota_m}
\arrow[bend right = 15]{dr}[swap]{\loc_J \circ \iota_m}
	&
		& \sab^{\Rop} \\
	& \sab^{\Rop}_J
	\arrow[bend right = 15]{ur}[swap]{\iota_j}
	\arrow[to=m, phantom, shift left]{}[rotate=90, pos = 0.40]{\Longleftarrow}
	\arrow[to=m, phantom, shift right = 0.5em, pos = 0.40]{}{\varphi}
		&
\end{tikzcd}
\end{equation*}
Then we have the following diagram
\begin{equation*}
\begin{tikzcd}[column sep = 4em, row sep = 1em]
\rgph
\arrow[bend left]{rr}[name=M]{\oN_{\max}}
\arrow[bend right]{rr}[name= P, swap]{\oN_+}
\arrow[from=M, to=P, phantom, shift right]{}[rotate=270]{\Longleftarrow}
	\arrow[from=M, to=P, phantom, shift left = 0.5em]{}{\eta}
	&
	& \sab^{\Rop}_m
	\arrow[bend left]{rr}[name=i]{\iota_m}
	\arrow[bend right =10]{dr}[swap]{\loc_J \circ \iota_m}
		&
			& \sab^{\Rop}
			\arrow{r}{H_{\bullet}}
				& \Ab^{\Rop} \\
	&
	&
		& \sab^{\Rop}_J
		\arrow[bend right = 10]{ur}[swap]{\iota_J}
		\arrow[to=i, phantom, shift left]{}[rotate=90, pos = 0.40]{\Longleftarrow}
		\arrow[to=i, phantom, shift right = 0.5em, pos = 0.40]{}{\varphi}
			&
				&
\end{tikzcd}
\end{equation*}
where the composition across the top is persistent homology and the composition across the bottom is magnitude homology. Alternatively, we can rearrange the important part (the left two-thirds) of the above diagram as
\begin{equation*}
\begin{tikzcd}[column sep = 8em, row sep = 4em]
\rgph
\arrow{r}[name = M]{\overline{N}_{+}}
\arrow[equal]{d}
	& \sab^{\Rop}_m
	\arrow{r}[name = i]{\iota_m}
	\arrow[equal]{d}
		& \sab^{\Rop}\
		\arrow[equal]{d}\\
\rgph
\arrow{r}[name = P]{\overline{N}_{\max}}
\arrow[from=M, to=P, phantom, shift right, pos = 0.55]{}[rotate=90]{\Longleftarrow}
\arrow[from=M, to=P, phantom, shift left = 0.5em, pos = 0.55]{}{\eta}
	& \sab^{\Rop}_m
	\arrow{r}[name = j]{\iota_J \circ \loc_J \circ \iota_m}
	\arrow[from=i, to=j, phantom, shift right, pos = 0.55]{}[rotate=90]{\Longleftarrow}
	\arrow[from=i, to=j, phantom, shift left = 0.5em, pos = 0.55]{}{\varphi}
		& \sab^{\Rop}
\end{tikzcd}
\end{equation*}
which makes clear the following refinement of Theorem \ref{thm:persvsmag}:

\begin{thmbis}{thm:persvsmag}
\label{thm:persvsmagprime}

Let $\eta$ and $\varphi$ as in the preceding, i.e. $\eta$ includes each $\ell^1$ complex into an $\ell^{\infty}$ complex and $\varphi$ is $J$-localization.

Then we have the following relationships:
\begin{equation*}
\begin{tikzcd}[row sep = 4em]
	& \textup{Global $\ell^1$ homology}
	\arrow[Rightarrow]{dl}[swap]{\eta}
	\arrow[Rightarrow]{dr}{\varphi}
		& \\
\textup{Persistent homology}
\arrow[Rightarrow]{dr}[swap]{\varphi}
	&
		& \textup{Magnitude homology}
		\arrow[Rightarrow]{dl}{\eta} \\
	& \textup{Local $\ell^{\infty}$ homology}
		&
\end{tikzcd}
\end{equation*}

\end{thmbis}

\section{Observations from the quantalic perspective}
\label{sec:apps}

 In this section we make some further observations naturally afforded us by the idea of considering different quantale structures $(\mcv, \otimes)$ in our framework.

\subsection{Ultrametric spaces}

An \emph{ultrametric space} is a special kind of (honest) metric space occurring in certain contexts, such as phylogenetic trees: it is an honest metric space $(X,d)$ which satisfies the stronger triangle inequality
\begin{equation*}
d(a,c) \leq \max \left ( d(a,b), d(b,c) \right ) \quad \text{for all} \; a,b,c \in X
\end{equation*}
Thus an ultrametric space is precisely a strict symmetric $\ell^\infty$ metric space, in our terminology. Then the next fact shows that persistent homology of an honest metric space is an indicator of its failure to be an ultrametric space.

\begin{proposition}
\label{prop:persult}

Let $X$ be an honest metric space. If $X$ is an ultrametric space, its persistent homology is trivial at all scales, in all dimensions $\geq 1$. That is, we have
\begin{equation*}
H_n \overline{N}_{\max} (X) (r) = 0
\end{equation*}
for all $r \in \Rop$ and $n \geq 1$.

\end{proposition}

\begin{proof}

Given any $r \in \Rop$, the relation of being at distance $\leq r$ in an ultrametric space $X$ is an equivalence relation on the points of $X$ (the stronger form of the triangle inequality is required for transitivity to hold). Thus $X$ may be partitioned into clusters of points, where the points in a single cluster are pairwise at distance $\leq r$ (and points from different clusters are at pairwise distance $> r$).

Then we have the following description of $N_{\max}(X)(r)$:
\begin{enumerate}

\item The vertex set $N_{\max}(X)(r)_0$ is the set of points of $X$;

\item For $a,b \in X$, there is exactly one $1$-simplex from $a$ to $b$ iff $a$ and $b$ are in the same cluster, and none otherwise.

\item For $n \geq 2$, for any $(n+1)$-tuple $(x_0, \dots, x_n)$ of points of $X$, there is exactly one $n$-simplex spanning those vertices (in that order) iff $x_0, \dots x_n$ all lie in the same cluster, and none otherwise.

\end{enumerate}
which makes it clear that the points of each cluster in $X$ comprise the vertex set of a contractible component of $N_{\max}(X)(r)$. Since the clusters partition $X$ we must have that $N_{\max}(X)(r)$ is the disjoint union of contractible components (with the number of components equal to the number of clusters), and thus the conclusion follows.

\end{proof}

The converse (trivial positive degree persistent homology implies ultrametric space) is false, as any metric space with $3$ points will necessarily have trivial positive degree persistent homology. It is unknown to the author whether some reasonable additional assumptions might guarantee the converse to Proposition \ref{prop:persult}.

\subsection{Local $\ell^p$ homology and approximate collinearity}

Let $X$ be an honest metric space. Given $p \in [1, \infty)$ and an ordered pair $(a,b)$ of distinct points $a,b \in X$, let us say that a point $c \in X$ \emph{$p$-interpolates} between $a$ and $b$ when $a \neq c \neq b$ and there exist $r, s \in \mathbb{R}$ such that $d(a,c) \leq r$, $d(c,b) \leq s$, and $\left ( r^p + s^p \right ) ^{\frac{1}{p}} = d(a,b)$. Some elementary real analysis shows that for $1 \leq p \leq q \leq \infty$, if a point $p$-interpolates between $a$ and $b$, then it also $q$-interpolates between $a$ and $b$.

When $p = 1$, we can take $d(a,c) = r$ and $d(c,b) = s$ in the above condition without loss of generality. That is, a point that $1$-interpolates between $a$ and $b$ is just a ``collinear'' point between them. From this perspective, we can see $p$-interpolation as a kind of ``approximate collinearity'' condition, where the approximation is better the closer $p$ is to $1$.

$X$ is said to be \emph{Menger convex} when for any pair of distinct points there exists a point that $1$-interpolates between them. For reasonably nice instances of honest metric spaces $X$, Menger convexity is equivalent to a more widely familiar notion of convexity (``geodesicity''), as detailed in e.g. \cite{convex}.

Now Theorem 7.4 in \cite{lsmaghom} exhibits magnitude homology (i.e. local $\ell^1$ homology) as an algebraic measure of (the failure of) convexity by showing that $\mch_1(X)(r) = H_1 \loc_J \oN_+(X)(r)$ is freely generated by the ordered pairs $(a,b)$ of distinct points $a, b \in X$ at distance $r$ for which there exists no $1$-interpolating point.

Recalling Proposition \ref{prop:presimp} above and straightforwardly extending the argument of the proof of Theorem 7.4 in \cite{lsmaghom} (which applies when $p = 1$) to more general choices of $p \in [1, \infty)$, we get the following (we exclude the case $p = \infty$ because the argument does not extend in that case):
\begin{proposition}
\label{prop:lsmaghom}

Let $(X,d)$ be an honest metric space and $p \in [1, \infty)$.

$H_1 \loc_J \oN_{+_p} (X)(r)$ is freely generated by ordered pairs $(a,b)$ of points in $X$ such that
\begin{enumerate}

\item $d(a,b) = r$, and;

\item There is no $p$-interpolating point between $a$ and $b$.

\end{enumerate}

\end{proposition}

 Thus for any pair $(a,b)$ of points in $X$, the quantity
\begin{equation*}
p_{a,b} = \inf \{ p \in [1, \infty] \mid (a,b) \text{ is trivial in } H_1 \loc_J \oN_{+_p} (X)(d(a,b)) \}
\end{equation*}
indicates the existence of approximately collinear points between $a$ and $b$ (where nonexistence gives $p_{a,b} = \infty$); a lower value of $p_{a,b}$ indicates better approximation to collinearity.

\subsection{Magnitude homology of automata}

We draw from aspects of automata theory found e.g. in \cite{betti}. A nondeterministic automaton with inputs from a monoid $M$ is an instance of a model of computation, in which (roughly) we have a set of states and rules by which elements of $M$ take each state to other states. We recall the following setup from \cite{betti} which exhibits a nondeterministic automaton as an enriched category.

Given a monoid $M$, we may take the free quantale $\mathcal{P}(M)$ whose objects are the subsets of $M$ and whose monoidal product is given by
\begin{equation*}
A \otimes B = \{ a \cdot b \mid a \in A \subseteq M, b \in B \subseteq M\}
\end{equation*}
where $\cdot$ is the monoid multiplication. The unit of $\mathcal{P}(M)$ is $\{e\}$ where $e$ is the identity element of $M$. A typical example for $M$ is the free word monoid on some fixed set of letters, with multiplication given by concatenation.

Then a nondeterministic automaton with inputs from $M$ is a category $X$ enriched over $\mathcal{P}(M)$. We are to view the objects of $X$ as the states, and the hom-object $X(a,b) \in \mathcal{P}(M)$ as the set of elements of $M$ which take the state $a$ to the state $b$. (The nondeterminism stems from the fact that we do not require that $X(a,b)$ be disjoint from $X(a,c)$ for $b \neq c$, so that an element of $M$ may take $a$ to more than one state.)

We note that $\mathcal{P}(M)$ is not an affine quantale, since $\{e\}$ is not terminal in $\mathcal{P}(M)$. Thus the material of our paper so far does not immediately apply in this framework. However, we may try the following. Regarding $M$ as a monoidal category, we may ask for a lax monoidal functor $\mbc: M \rightarrow (\mathbb{R},+)$ which acts as a ``cost function''. That is, we associate a nonnegative real number $\mbc(m)$ to each element $m \in M$ which represents some kind of cost for acting by $m$, and the lax monoidality of $\mbc$, i.e.
\begin{enumerate}

\item $\mbc(m_2) + \mbc(m_1) \geq \mbc(m_2 \cdot m_1)$;

\item $0 \geq \mbc(\{e\})$ (which implies $\mbc(\{e\}) = 0$)

\end{enumerate}
is just the condition that
\begin{enumerate}

\item the cost of acting by $m_1$ and then by $m_2$ should be at most the sum of the individual costs of $m_1$ and $m_2$, and;

\item the cost of not acting at all is $0$.

\end{enumerate}

Given such a cost function $\mbc: M \rightarrow \mathbb{R}$, we have an induced cost function (lax monoidal functor) $\mbC: \mathcal{P}(M) \rightarrow \mathbb{R}$ given by
\begin{equation*}
\mbC(A) = \inf\limits_{a \in A} \mbc(a)
\end{equation*}
which is strong monoidal if $c$ is.

Thus given any $\mathcal{P}(M)$-category $X$, we get an induced $(\mathbb{R},+)$-category $\mbC(X)$ with the same objects and hom-objects $\mbC(X)(a,b) = \mbC(X(a,b))$. In the context of considering an automaton as a $\mathcal{P}(M)$-category $X$, $\mbC(X)$ is the (generalized) metric space whose points are the states of the automaton and each hom-object $\mbC(X)(a,b)$ is the optimal cost of going from state $a$ to state $b$. Given an ordered pair $(a,b)$ of states, let us say that it is \emph{cost-primitive} if there exists no state $c$ such that $\mbC(X)(a,b) = \mbC(X)(a,c) + \mbC(X) (c,b)$. It is evident that $(a,b)$ is cost-primitive iff either there is no ``intermediate state'' $c$ for which there exist $m_1 \in X(a,c)$ and $m_2 \in X(c,b)$, or for every such intermediate state $c$ and for every way of getting from $a$ to $b$ through $c$, there is a way to get from $a$ to $b$ which costs strictly less.

Let us denote the functor $H_{\bullet} \loc_J \oN_+: \rgph \rightarrow \Ab^{\Rop}$ by $\mch_{\bullet}: \rgph \rightarrow \Ab^{\Rop}$, so that $\mch_{\bullet}$ just takes the magnitude homology of each $\mathbb{R}$-graph. If we have that the induced $(\mathbb{R},+)$-category $\mbC(X)$ is strict (i.e. the automaton had no pairs of distinct states connected by a $0$-cost transition), then following the reasoning of (the proof of) Theorem 7.4 of \cite{lsmaghom}, we have that for each $r > 0$, $\mch_1 \left ( \mbC(X) \right ) (r)$ is the free abelian group on the set of ordered pairs $(a,b)$ of states which are cost-primitive and for which $\mbC(X)(a,b) = r$.

As for why this application of magnitude homology might be interesting, let us assume the setup above and that $(a,b)$ is a cost-primitive pair of states. Then either there is no intermediate state between them, or for any way of getting from $a$ to $b$ through an intermediate state there is a cheaper way to get from $a$ to $b$. In the former case there is no more to say, so let us assume the latter. Let us assume also that the cost function $\mbc$ (and thus also the induced cost function $\mbC$) is strong monoidal - this is the case, for example, if $M$ is the free word monoid on some fixed set of letters and $\mbc(m)$ depends linearly on the length of $m \in M$. Then for each state $c$ there must be some $m \in X(a,b)$ such that $\mbc(m) < \mbc(m_2 \cdot m_1) = \mbc(m_1) + \mbc(m_2)$ for all $m_1 \in X(a,c)$ and $m_2 \in X(c,b)$. Assuming that the number of possible states is finite (as in any reasonable model of computation) and that $\mbC(X)(a,b) < \infty$, this implies that there is at least one way to transition from state $a$ to $b$ which is cheaper than all possible composite state transitions taking $a$ to $b$.

Conversely, if $r>0$ and $(a,b)$ is a pair of states for which $\mbC(X)(a,b) = r$, and there is at least one state transition $m$ from $a$ to $b$ which is strictly cheaper than all possible composite state transitions which take $a$ to $b$, then we must have that $\mbc(m) < \mbc(m_2 \cdot m_1) = \mbc(m_1) + \mbc(m_2)$ for all $c \in X$, $m_1 \in X(a,c)$, and $m_2 \in X(c,b)$. If $\mbc$ ``has discrete range'', i.e. if there is no $t \in \mathbb{R}$ for which it is simultaneously true that $t$ is the infimum of a set of possible values of $\mbc$ and that $t$ lies outside the image of $\mbc$, then the previous statement implies that $\mbC(X)(a,b) \leq \mbc(m) < \mbC(X)(a,c) + \mbC(X)(c,b)$ for all possible $c$. Thus $(a,b)$ is cost-primitive and a generator of $\mch_1 \left ( \mbC(X) \right ) (r)$.

We summarize this analysis as follows.

\begin{proposition}
\label{prop:magaut}

Let $M$ be a monoid and $\mathcal{P}(M)$ the corresponding free quantale.

Let $\mbc: M \rightarrow (\mathbb{R},+)$ be a strong monoidal functor, and $\mbC: \mathcal{P}(M) \rightarrow (\mathbb{R},+)$ the corresponding induced strong monoidal functor.

Let $X$ be an automaton with inputs from $M$, i.e. a $\mathcal{P}(M)$-category $X$ with finitely many objects. We have the corresponding $(\mathbb{R},+)$-category $\mbC(X)$.

If $\mbC(X)$ is strict (i.e. there is no pair of distinct states in $X$ connected by a $0$-cost transition), then for $r > 0$, if $(a,b)$ is a generator of $\mch_1 \left ( \mbC(X) \right ) (r)$ then there is at least one state transition $m$ taking $a$ to $b$ which does not occur as a composite of state transitions, and the cost of $m$ is strictly less than the cost of any such composite transition.

If additionally $\mbc$ has discrete range (in the sense described previously), then the converse is true: for $(a,b)$ with $\mbC(X)(a,b)=r$, if there is at least one $m \in X(a,b)$ such that $\mbc(m)$ is less than the cost of all possible composite transitions from $a$ to $b$, then $(a,b)$ is cost-primitive and thus a generator of $\mch_1 \left ( \mbC(X) \right ) (r)$.

\end{proposition}

We note that even if $\mbC(X)$ is not strict, we can collapse all the points at distance $0$ (in some careful way involving symmetrization of distances or application of $\Fr: \rgph \rightarrow \rpcat$) and then apply the result. In that case Proposition \ref{prop:magaut} above is a result about equivalence classes of pairs of states, rather than actual pairs of states.

As a final remark, we note that one could just as well take the persistent homology $H_{\bullet} \oN_{\max}$ of $C(X)$ instead of magnitude homology, but it is unclear if the resulting homology groups have a useful interpretation in this context of automata.

\bibliography{References}
\bibliographystyle{alpha}

\end{document}